\documentclass[a4paper,11pt,reqno]{amsart}          

\usepackage{amsfonts,amsmath,latexsym,amssymb} 
\usepackage{amsthm}      
\usepackage[dvipsnames]{xcolor}          
\usepackage{mathrsfs,upref}

\usepackage[foot]{amsaddr}
\usepackage[left=30mm,right=30mm,top=30mm,bottom=30mm]{geometry}
\usepackage{hyperref}
\usepackage{cite}

\usepackage{float}
\usepackage{tikz}
\usepackage{circuitikz}
\usepackage{hyperref}
\usepackage{mathtools}
\usepackage{enumitem}
\usepackage{pgfplots}
\pgfplotsset{compat=1.18}

\allowdisplaybreaks

\theoremstyle{plain}
\newtheorem{theorem}{Theorem}[section]
\newtheorem*{theorem*}{Theorem}
\newtheorem{lemma}[theorem]{Lemma}
\newtheorem*{lemma*}{Lemma}

\theoremstyle{remark}
\newtheorem{remark}[theorem]{Remark}
\newtheorem*{remark*}{Remark}

\newtheorem*{example*}{Example}
\theoremstyle{definition}

\numberwithin{equation}{section}

\DeclareMathOperator{\sgn}{sgn}

\renewcommand{\L}{\mathcal{L}}

\renewcommand{\H}{\mathscr{H}}
\renewcommand{\L}{\mathscr{L}}
\newcommand{\E}{\mathscr{E}}
\newcommand{\V}{\mathscr{V}}
\newcommand{\Z}{\mathbb{Z}}
\newcommand{\R}{\mathbb{R}}
\newcommand{\C}{\mathbb{C}}
\newcommand{\ceq}{\coloneqq}
\newcommand{\ds}{\displaystyle}
\newcommand{\wt}{\widetilde}

\title[Periodic discrete graphs with prescribed spectrum]{Periodic discrete graphs with prescribed spectrum}

\author[Andrii Khrabustovskyi]{Andrii Khrabustovskyi$^1$}
\address{$^1$\,Department of Physics, Faculty of Science, University of Hradec Kr\'{a}lov\'{e}, Rokitansk\'eho 62, 50003 Hradec Kr\'alov\'e, Czech Republic}
\email{andrii.khrabustovskyi@uhk.cz}

\author[Anna Muranova]{Anna Muranova$^2$}
\address{$^2$\,Faculty of Mathematics and Computer Science, University of Warmia and Mazury in Olsztyn, ul. Słoneczna 54, 10-710 Olsztyn, Poland} 
\email{anna.muranova@uwm.edu.pl}

\begin{document}

\begin{abstract}
We construct a periodic weighted graph whose discrete Laplacian has a spectrum with precisely $n$ gaps. Moreover, we show  that by an appropriate choice of the weights, the endpoints of these gaps, as well as the upper edge of the spectrum, attain the prescribed values.
The underlying graph has a brush-like geometry: it consists of an infinite  chain of vertices, each of which is connected to $n$ additional pendant vertices by extra edges.
Semi-explicit formulae for the weight coefficients are provided: some of the coefficients are determined explicitly, while others are given as roots of an explicitly determined polynomial.
\end{abstract}


\maketitle
 
\thispagestyle{empty}

\section{Introduction}
\label{sec1}

The topic of this paper lies within the spectral theory of Laplacians on periodic discrete graphs. These operators act on functions defined on the vertices as difference operators, while the edges play a secondary role, serving merely as labels that encode the connectivity between vertices.
This is an active area of research that brings together several branches of mathematics and has a wide range of applications, e.g. in  nano-technology, crystallography, optics and chemistry. We refer the reader to the monographs \cite{BK12,KN22,KLW21,BH12, Ha99, Sun13} and references therein.

In this paper, we study Laplacians on \emph{periodic discrete graphs}.
For such operators, Floquet--Bloch theory (see, e.g., \cite[Chapter~4]{BK12} for its version for discrete graphs) 
implies that the spectrum exhibits the so-called \emph{band-gap structure}, that is, it consists of a union of finitely many closed finite intervals (bands) possibly separated by open intervals (gaps). In general, however, gaps are not guaranteed to occur, since the bands may overlap, causing the corresponding gaps to close. For example, the spectrum of the combinatorial discrete Laplacian
$$
\Delta:\ell^2(\mathbb Z)\to\ell^2(\mathbb Z), \qquad
(\Delta f)(i) = 2f(i)-f(i-1)-f(i+1),\quad i\in\Z
$$
on the infinite chain of vertices $\{\dots,-2,-1,0,1,2,\dots\}$ has no gaps,
namely
$\sigma(\Delta)=[0,4].$

The existence of spectral gaps is important in various  applications, since the spectrum encodes the transport properties of the underlying medium, and spectral gaps correspond to ranges of wavelengths for which wave propagation through the medium is prohibited. We refer the reader to \cite{F93} for applications to dielectric materials and to \cite{CGPNG09, NG04,SDD98, SW58} for connections between the spectral properties of graph operators and the behavior of nanomaterials, especially with regard to photonic band gaps
(see, e.g., \cite{BDE93}).

A method for opening spectral gaps was proposed by Schenker and Aizenman in \cite{SA00}.
The idea is to start with a fixed graph and ``decorate'' it by ``gluing'' to each vertex a copy of a given finite graph.
Note that this method is originally not restricted to periodic graphs.
See also \cite{Ku05}, where a similar idea was applied to differential operators on metric graphs (the so-called \emph{quantum graphs}).
\smallskip

In the current paper we investigate spectral properties of some specific class of periodic  graphs with weights, i.e. additionally to
the combinatorial structure (vertices linked by edges)
we equip a graph $\Gamma$ with two positive 
functions $\mu$ and $m$ defined on the set of edges
and  on the set of vertices, respectively.
The graph has a brush-like geometry: it consists of an infinite  chain of vertices, each of which is connected to $n$ additional pendant vertices by extra edges -- see Figure~\ref{fig1}.
The Laplacian $\L$ on the weighted graph $\Gamma$  
acts on   $f:\V\to\C$, where $\mathscr{V}$ is the vertex set of $\Gamma$, as follows:
\begin{gather*}
(\L f)(v)=(m(v))^{-1}\sum\mu(e)(f(v) - f(u)) ,
\quad v\in\V.
\end{gather*}
where the sum above is taken over all vertices $u$ adjacent to $v$, with $e$ denoting the edge connecting $u$ and $v$.

The main peculiarity of these graphs is that their spectral gaps can be effectively controlled through an appropriate choice of the weight functions. Firstly, we show (Theorem~\ref{th1}) that, under rather mild assumptions on the weights (see \eqref{aa}), the spectrum of $\L$ possesses exactly $n$ gaps. Our main result (Theorem~\ref{th2}) addresses the  {inverse problem}: 
\emph{we prove that the gap edges, as well as the upper edge of the spectrum, can be prescribed arbitrarily and realized by a suitable choice of the weights}. 

Semi-explicit formulae for the weights realizing the prescribed spectral structure are obtained. By this we mean that some of the weight coefficients are not given in closed form, but  as roots of a certain polynomial whose coefficients are explicitly expressed in terms of the prescribed gap edges.

For periodic quantum graphs, this type of inverse problem was studied by Barseghyan and the first author in \cite{BK15}. The result obtained there is asymptotic in nature: a sequence of operators was constructed whose first $n$ spectral gaps converge to prescribed intervals. At the end of the paper, in Remark~\ref{rem:comparison}, we compare the results of \cite{BK15} with those of the present paper.
\smallskip

The paper is organized as follows. In the next section, we introduce the weighted graph and the associated Laplacian, and describe its spectrum. In Section~3 we prove some auxiliary technical lemmas. In Section~4 we investigate several properties of the spectrum. The inverse problem  is then solved in Section~5.

\section{Graph $\Gamma$, operator $\mathscr{H}$ and its spectrum}
\label{sec2}

Let $\Gamma=(\V,\E)$ be a  graph with vertex set
$
\V=\{v_{i,j}\,:\ i\in\mathbb Z,\ j=0,\dots,n\}
$
and edge set
$
\E=\{e_{i,j}\,:\ i\in\mathbb Z,\ j=0,\dots,n\}.
$
The edges connect the vertices as follows:
\begin{gather}\label{graph:structure}
v_{i,0}\overset{e_{i,0}}\sim v_{i+1,0},
\quad i\in\mathbb Z,
\qquad
v_{i,0}\overset{e_{i,j}}\sim v_{i,j},
\quad i\in\mathbb Z,\ j=1,\dots,n.
\end{gather}
Here, the notation
$
u\overset{e}\sim v
$
means that the vertices $u$ and $v$ are connected by the edge $e$.
The graph is depicted in Figure~\ref{fig1}.
Note that we consider the edges  
as relations between vertices, rather then physical or geometrical links.

We equip the graph $\Gamma$ with two  
weight functions $\mu:\E\to \R$ and $m:\V\to \R$   defined by
\begin{gather}\label{mu:m}
	\mu(e_{i,j})=\mu_j,\quad i\in\Z,\ j=0,\dots,n,
	\qquad
	m(v_{i,j})=m_j,\quad i\in\Z,\ j=0,\dots,n,
\end{gather}
where $\mu_j,\, m_j$, $j=0,\dots,n$, are positive numbers.

\begin{figure}[ht]
\centering
\begin{tikzpicture}[
    scale=1,
    vertex/.style={circle,fill=black,inner sep=1.6pt},
    every node/.style={font=\small}
]

\foreach \i in {0,1,2, 3}{
    \node[vertex] (v0\i) at (3*\i,0) {};
    \node at ($(v0\i)+(0,0.18)$) {$v_{\i,0}$};
}

\foreach \i in {0,1,2, 3}{

    \node[vertex] (v1\i) at ($(v0\i)+(-120:2)$) {};
    \node at ($(v1\i)+(0,-0.23)$) {$v_{\i,1}$};

    \node[vertex] (v2\i) at ($(v0\i)+(270:2.2)$) {};
    \node at ($(v2\i)+(0,-0.23)$) {$v_{\i,2}$};

    \node[vertex] (v3\i) at ($(v0\i)+(-60:2)$) {};
    \node at ($(v3\i)+(0,-0.23)$) {$v_{\i,3}$};
}

\foreach \i/\k in {0/1,1/2, 2/3}{
    \draw (v0\i) --
        node[above=0pt,pos=0.5] {$e_{\i,0}$}
        (v0\k);
}

\foreach \i in {0,1,2,3}{

    \draw (v0\i) --
        node[pos=0.52,left=-2pt] {$e_{\i,1}$}
        (v1\i);

    \draw (v0\i) --
        node[pos=0.5,right=-2pt] {$e_{\i,2}$}
        (v2\i);

    \draw (v0\i) --
        node[pos=0.52,right=-2pt] {$e_{\i,3}$}
        (v3\i);
}

\node at (-2,0) {$\cdots$};
\node at (11,0) {$\cdots$};

\node at (-2,-1.8) {$\cdots$};
\node at (11,-1.8) {$\cdots$};

\end{tikzpicture}
\caption{The graph $\Gamma$ for $n=3$.}
\label{fig1}
\end{figure}
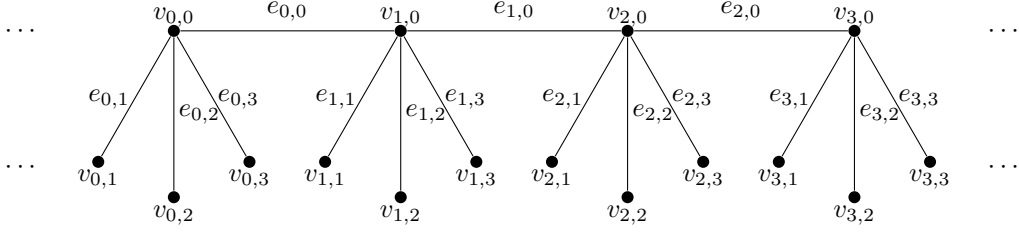

Now, we  introduce the Laplacian $\L$ on the weighted graph $\Gamma$ (cf.~\cite[Definition~1.2.4]{BK12} or \cite[Section~2.2]{KN22}).
We define the Hilbert space $\H$ by
$$
\H=\Big\{f:\V\to\C:\quad\sum_{i\in\Z}\sum_{j=0}^n m_j|f(v_{i,j})|^2<\infty\Big\}
$$ 
equipped with the inner product 
$(f,g)_{\H}= \sum_{i\in\Z}\sum_{j=0}^n m_j f(v_{i,j})\overline{g(v_{i,j})}.$ 
Then we introduce the operator $\L:\H\to \H$ by
\begin{gather}
\label{L:action}
(\L f)(v)=\frac{1}{m(v)}\sum_{u\overset{e}\sim v}\mu(e)(f(v) - f(u)),\quad v\in\V.
\end{gather}
The operator $\L$ is bounded, non-negative, and self-adjoint.

We denote:
\begin{gather}\label{ab:notations} 
a_j\ceq \mu_j/m_j,\ j=1,\dots,n,\qquad 
b_j\ceq \mu_j/m_0,\ j=0,\dots,n.
\end{gather}
We assume that the numbers $a_j$ are pairwise distinct. Without loss of generality, we may therefore order them so that
\begin{gather}\label{aa}
	a_j<a_{j+1},\quad j=1,\dots,n-1.
\end{gather}
 
To describe the spectrum of the operator $\L$ we introduce a rational function
\begin{gather}\label{F}
	F(x)=x\left(\sum_{j=1}^n\frac{b_j}{a_j-x}+1\right).
\end{gather} 
The function $F(x)$ is strictly increasing on each interval $(a_j,a_{j+1})$, as well as on $(-\infty,a_1)$ and on $(a_n,\infty)$. Moreover, one has
$$
\lim_{x\to\pm\infty}F(x)=\pm\infty,\quad
\lim_{x\to a_j\pm 0}F(x)=\mp\infty,\quad
F(0)=0.
$$
Consequently, $F(x)$ has exactly $n+1$ roots $s_j\in[0,\infty)$, $j=0,\dots,n$,
which we enumerate in ascending order. Similarly, the equation
\begin{gather}\label{t:eq}
F(x)=4b_0
\end{gather}
has exactly $n+1$ positive roots $t_j$, $j=1,\dots,n+1$,
which we also enumerate in ascending order.
Clearly (see Figure~\ref{fig2}),
\begin{gather}
\label{interlace}
0=s_0,\qquad
s_{j-1}<t_j<a_j<s_j,\ j=1,\dots,n,\qquad
s_n<t_{n+1}.
\end{gather}

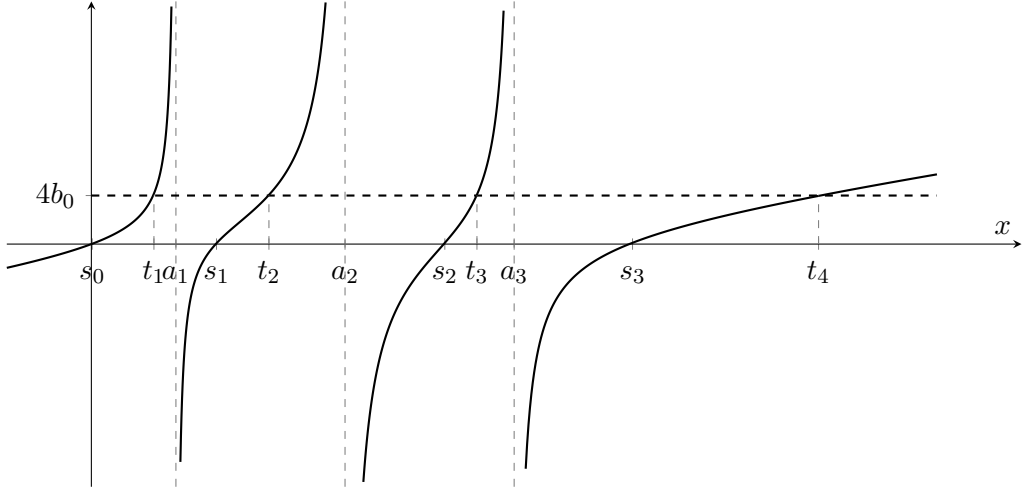
\begin{figure}[ht]
	\centering
\begin{tikzpicture}
	
	\begin{axis}[
		width=15cm,
		height=8cm,
		axis lines=middle,
		xlabel={$x$},
		xmin=-1,
		xmax=11,
		ymin=-20,
		ymax=20,
		samples=300,
		restrict y to domain=-20:20,
		xtick={0.01,0.74,1,1.48,2.1,3,4.18,4.56,5,6.4,8.6},
		xticklabels={$s_0$,$t_1$,$a_1$,$s_1$,$t_2$,$a_2$,$s_2$,$t_3$,$a_3$,$s_3$,$t_4$},
		xticklabel style={
			text height=1.5ex,
			text depth=.25ex
		},
		ytick={4},
		yticklabels={$4b_0$}
		]
		
		\addplot[dashed,gray] coordinates {(1,-20) (1,20)};
		\addplot[dashed,gray] coordinates {(3,-20) (3,20)};
		\addplot[dashed,gray] coordinates {(5,-20) (5,20)};
		
		\addplot[black,thick,domain=-1:0.98]
		{x*(1 + 1/(1-x) + 1.5/(3-x) + 0.5/(5-x))};
		
		\addplot[black,thick,domain=1.02:2.98]
		{x*(1 + 1/(1-x) + 1.5/(3-x) + 0.5/(5-x))};
		
		\addplot[black,thick,domain=3.02:4.98]
		{x*(1 + 1/(1-x) + 1.5/(3-x) + 0.5/(5-x))};
		
		\addplot[black,thick,domain=5.02:10]
		{x*(1 + 1/(1-x) + 1.5/(3-x) + 0.5/(5-x))};
		
		\addplot[dashed,thick,domain=0:10]{4}; 
		
		\addplot[dashed,gray] coordinates {(0.74,0) (0.74,4)};
		\addplot[dashed,gray] coordinates {(2.1,0) (2.1,4)};
		\addplot[dashed,gray] coordinates {(4.56,0) (4.56,4)};
		\addplot[dashed,gray] coordinates {(8.6,0) (8.6,4)};
		
	\end{axis}
	
\end{tikzpicture}
\caption{The function $F(x)$ for $n=3$}
\label{fig2}
\end{figure}

We are now in the position to formulate the first result.

\begin{theorem}\label{th1}
	One has:
	\begin{gather}\label{th1:eq}
	\ds\sigma(\L)=
	\bigcup_{j=1}^{n+1} (s_{j-1},t_j)=
	[0,t_{n+1}]\setminus\Big( \bigcup_{j=1}^n (t_j,s_j)\Big).
	\end{gather}
\end{theorem}

\begin{proof}
Since the graph $\Gamma$ is $\Z$-periodic and the  operator $\L$ commutes with the $\Z$-action, the standard Floquet--Bloch theory applies 
(see \cite[Chapter~4]{BK12} for more details). In particular, the spectrum $\sigma(\L)$ has a band-gap structure: it is a union of closed finite intervals $I_j$ (spectral bands) determined by the eigenvalues of the fiber operators $\L(k)$, $k\in[0,2\pi)$, acting in the space $\H(k)$  of $k$-quasi-periodic functions $f:\Gamma\to\C$, that is,
$$f(v_{i,j})=e^{\mathrm{i} k}f(v_{i-1,j}),\quad i\in\Z,\ j=0,\dots,n,$$
with the action of $\L(k)$ given by \eqref{L:action}.
The quasiperiodicity condition implies that  the spectral problem 
\begin{gather}\label{Lf=lambdaf}
	\L(k)f=\lambda f,\quad 0\not=f\in\H(k)
\end{gather}
reduces to the finite-dimensional  problem
\begin{gather}\label{Lf=lambdaf:matrix}
L(k)f=\lambda f,\quad 0\not=f=(f_0,\dots,f_n)^T\in\C^{n+1},
\end{gather}
where $L(k)$ is the $(n+1)\times(n+1)$  matrix given by 
$$
L(k)=
\begin{pmatrix}
	b_0(1-e^{\mathrm{i}k})+b_0(1-e^{-\mathrm{i}k}) + \displaystyle\sum_{j=1}^n b_j &\quad -b_1 & -b_2 & \cdots & -b_n \\
	-a_1 &\quad a_1 & 0 & \cdots & 0 \\
	-a_2 &\quad 0 & a_2 & \cdots & 0 \\
	\vdots &\quad \vdots & \vdots & \ddots & \vdots \\
	-a_n &\quad 0 & 0 & \cdots & a_n
\end{pmatrix}
$$
(recall that $a_j$, $j=1,\dots,n$ and $b_j$, $j=0,\dots,n$ are defined by \eqref{ab:notations}).
To obtain \eqref{Lf=lambdaf:matrix} one  writes the equation \eqref{Lf=lambdaf} at  the vertices 
$v_{0,j}$, $j=0,\dots,n$, which form a fundamental domain of $\Gamma$,
and use the quasiperiodicity conditions
$$f(v_{1,0})=f(v_{0,0})e^{\mathrm{i}k},\qquad f(v_{-1,0})=f(v_{0,0})e^{-\mathrm{i}k}.$$
The matrix $L(k)$ is symmetric in $\C^{n+1}$ equipped with the weighted inner product
$$(f,g)_{\C^{n+1}}=\sum_{j=0}^n m_j f_j\overline{g_j}=
m_0 \Big(f_0\overline{g_0}+\sum_{j=1}^n b_j a_j^{-1}\overline{g_j}\Big).$$
We denote its eigenvalues, ordered increasingly and counted with multiplicities, by
$$\lambda_1(k)\leq \lambda_2(k)\le\dots\le \lambda_{n+1}(k).$$ 
For each fixed $j$,   $\lambda_j(k)$ depends continuously 
on $k\in [0,2\pi]$. Consequently, the set 
$$
I_j\ceq \bigcup_{k\in [0,2\pi]}\lambda_j(k),
$$
is a compact interval.
The Floquet-Bloch theory yields the representation
\begin{gather}\label{Floquet-Bloch}
	\sigma(\L)=\bigcup_{j=1}^{n+1} I_j.
\end{gather} 
Therefore, the description of the spectrum of $\L$ reduces to the analysis of the eigenvalues of the matrices $L(k)$.
 
Taking into account \eqref{aa}, 
one can easily show that if $\lambda= a_j$ for some $j\in\{1,\dots,n\}$, then 
the   problem \eqref{Lf=lambdaf:matrix}  possesses only a zero solution $f$.
Thus, 
\begin{gather}\label{lambdanotin}
\lambda\notin \{a_j,\ j=1,\dots,n\}.
\end{gather} 
The matrix equality \eqref{Lf=lambdaf:matrix} 
yields the system of $n+1$ equations for unknowns $f_j$, $j=0,\dots,n$. 
From the last $n$ equations of this systems we deduce, taking into account \eqref{lambdanotin}:
\begin{gather}\label{fj}
f_j=\frac{a_j f_0}{a_j-\lambda},\quad j=1,\dots,n.
\end{gather}
Substituting \eqref{fj} into the first equation and taking into account that 
$f_0\not=0$ (since otherwise \eqref{fj} would imply  $f_j=0$ for all $j=1,\dots,n$),
we arrive, after straightforward computations, at the following equation for $\lambda$:
\begin{gather*}
4b_0\sin^2{k\over 2}=F(\lambda),
\end{gather*}
where the function $F$ is given by \eqref{F}. 
Thus, the eigenvalues $\lambda_j(k)$ correspond to the intersection points of the graph of $F(x)$ with the horizontal line $y=4b_0\sin^2{k\over 2}$.
The properties of the function $F(x)$ (see Figure~\ref{fig2}) imply that
$$
\forall k\in [0,2\pi]:\quad \lambda_{j}(0)\le \lambda_{j}(k)\le \lambda_j(\pi),\quad j=1,\dots,n.
$$
By definition of the numbers $s_j$ and $t_j$, we get 
$\lambda_j(0)=s_{j-1}$ and $\lambda_j(\pi)=t_{j}$.
Hence, $I_j=[s_{j-1},t_j]$, which together with 
\eqref{Floquet-Bloch} implies the first equality in \eqref{th1:eq}.
The second equality  follows from  \eqref{interlace}. 
The theorem is proven.
\end{proof}

Theorem~\ref{th1} demonstrates that the operator $\L$ has $n$ gaps  $(t_j,s_j)$, $j=1,\dots,n$. 
Our main objective is to solve the \emph{inverse problem}: to prove that, through a suitable choice of the coefficients $\mu_j$ and $m_j$, the positions of the gap endpoints, as well as the upper spectral edge $t_{n+1}$, can be made to coincide with prescribed values.\smallskip

The rest of the paper is organized as follows.
In Section~\ref{sec3}, we prove two auxiliary technical lemmas concerning a polynomial whose roots are the numbers $a_j$.
In Section~\ref{sec4}, we derive several useful relations linking $a_j$, $b_j$, and $s_j$, $t_j$. Finally, in Section~\ref{th2}, we formulate and prove the main result of the paper, Theorem~\ref{th2}, which solves the inverse problem stated above.

\section{Auxiliary lemmas}
\label{sec3}

To prove the main result of this work, we will need two technical lemmas stated below. Since their proof is rather lengthy and technical, we devote a separate section to them. 

\begin{lemma}\label{lemma:Q}
Let $B>0$.
Let $s_j$, $j=0,\dots,n$ and $t_j$, $j =1,\dots, n+1$ be given  numbers satisfying 
\begin{gather}\label{st}
0=s_0< t_1,\quad t_j< s_j< t_{j+1},\ j=1,\dots,n.
\end{gather}
Define the polynomial  $Q(x)$  of degree $n$   by
\begin{gather}\label{Qdef}
Q(x)=\prod_{i=1}^n(t_i-x)-\dfrac 1B\sum_{j=1}^n \left( \prod_{m=0}^n (s_m-t_j) \prod_{\substack{i=1\\i\ne j}}^n\dfrac{x-t_i}{t_j-t_i}\right).
\end{gather}
Then
 \begin{equation}\label{eq::Qslstatement}
\forall \ell\in\{0,\dots,n\}:\quad Q(s_\ell)=\left(\prod_{i=1}^n (t_i-s_\ell)\right)\left(1+\dfrac{1}{B}\left(\sum_{\substack{i=1\\i\ne \ell}}^n s_i-\sum_{\substack{i=1}}^n t_i\right)\right).
\end{equation}
\end{lemma}
 
\begin{remark}
The role of the above lemma, which at first glance may appear somewhat unrelated to our problem, is the following. We will show later (see Lemma~\ref{lemma:properties}(i)) that if $s_j$ and $t_j$ are defined as in Theorem~\ref{th1}, then the numbers $a_j$ are precisely the roots of the polynomial $Q(x)$ with $B=4b_0$. This observation will allow us to solve the inverse problem by choosing $a_j$ as the roots of $Q(x)$. However, the numbers $a_j$ chosen in this way must satisfy \eqref{interlace}, and the proof of this fact relies essentially on Lemma~\ref{lemma:Q}.
\end{remark} 
 
\begin{remark}\label{rem:alter}
It is easy to see that the leading coefficient of the polynomial $Q(x)$ is equal to $(-1)^n$ and that
$Q(t_j)= {-P(t_j)}/{B}$ for $j=1,\dots,n,$
where $P(x)$ is a polynomial of degree $n+1$ given by
\begin{gather}
\label{Pdef}
P(x)\ceq \prod_{m=0}^n (s_m-x).
\end{gather}
This follows from the fact that  the second term in the right-hand-side of \eqref{Qdef} is precisely the Lagrange interpolation polynomial of degree  $n-1$ passing through the points $(t_j,{-P(t_j)}/{B})$, $j=1,\dots,n$).
These two properties (the prescribed leading coefficient and the interpolation data at $n$ distinct points) uniquely determine the polynomial $Q$.
We will use this alternative characterization of $Q$ later on.
\end{remark}

\begin{proof}[Proof of Lemma~\ref{lemma:Q}]
Instead of $Q(x)$, we will work with the function
$$\widehat{Q}(x)\ceq Q(x)\prod_{1\le i<k\le n}(t_k-t_i)$$
This choice is motivated by the fact (see below) that it eliminates the factors $(t_j-t_i)$ appearing in the denominator of \eqref{Qdef}. 

It is easy to see that 
$$
\forall j\in\{1,\dots,n\}:\quad
\prod_{1\le i<k\le n}(t_k-t_i)=(-1)^{n-j}
\prod_{\substack{i=1\\i\ne j}}^n {(t_j-t_i)}
\prod_{\substack{1\le i<k\le n\\i\ne j,k\ne j}}(t_k-t_i)
$$
Using the above equality, we get for any fixed $\ell$, $0\le \ell\le n$:
\begin{align*}
\widehat{Q}(s_\ell)
&=\prod_{i=1}^n(t_i-s_\ell)\prod_{1\le i<k\le n}(t_k-t_i)
\\
&+\sum_{j=1}^n  \dfrac{(-1)^{n-j+1}}{B}\left(\prod_{m=0}^n(s_m-t_j)\prod_{\substack{i=1\\i\ne j}}^n {(s_\ell-t_i)}\prod_{\substack{1\le i<k\le n\\i\ne j, k\ne j}}(t_k-t_i)\right)\\
&=\prod_{i=1}^n(t_i-s_\ell)\left(\prod_{1\le i<k\le n}(t_k-t_i)+\sum_{j=1}^n  \dfrac{(-1)^{j-1}}{B}\left({\prod_{\substack{m=0\\m\ne \ell}}^n(s_m-t_j)}\prod_{\substack{1\le i<k\le n\\i\ne j, k\ne j}}(t_k-t_i)\right)\right).
\end{align*}
Thus
\begin{equation}\label{eq::wQsl}
\widehat{Q}(s_\ell)=\prod_{i=1}^n(t_i-s_\ell)\left(\prod_{1\le i<k\le n}(t_k-t_i)+F^{(\ell)}(t_1,\dots,t_n)\right),
\end{equation}
where $F^{(\ell)}(x_1,\dots, x_n)$ is a polynomial in $n$ variables given by
\begin{gather}\label{sumFj}
F^{(\ell)}(x_1,\dots, x_n)\ceq\sum_{j=1}^n F_j^{(\ell)}(x_1,\dots x_n),
\end{gather}
with
\begin{equation}\label{Fj}
F_j^{(\ell)}(x_1,\dots, x_n)\ceq\dfrac{(-1)^{j-1}}{B}{\prod_{\substack{m=0\\m\ne \ell}}^n(s_m-x_j)}\prod_{\substack{1\le i<k\le n\\i\ne j, k\ne j}}(x_k-x_i).
\end{equation}

It follows easily from the definition \eqref{Fj} of $F_j^{(\ell)}$ that 
for any fixed $p,r\in\{1,\dots,n\}$ with $p\ne r$, one has
$$
 F_j^{(\ell)}(\sigma(x_1),\dots \sigma(x_n))=\begin{cases}
-F_j^{(\ell)}(x_1,\dots x_n),& j\notin\{ p, r\}\\
-F_r^{(\ell)}(x_1,\dots x_n),& j=p,\\
-F_p^{(\ell)}(x_1,\dots x_n),& j=r,
 \end{cases}
 $$
where $\sigma(x_i)=x_i, i\notin\{ p, r\}$, $\sigma(x_r)=x_p$, $\sigma(x_p)=x_r$.
 Consequently, $F^{(\ell)}(x_1,\dots, x_n)=\sum_{j=1}^n F_j^{(\ell)}(x_1,\dots x_n)$ is an alternating polynomial.
It is well-known (see, e.g., \cite{Br08})  that any alternating polynomial is a product of the Vandermonde polynomial 
$$W(x_1,\dots x_n)\ceq \prod_{1\le i<k\le n}(x_k-x_i)$$ and a symmetric polynomial.   
Straightforward computations yield 
 $$
 \deg F^{(\ell)}=\deg F_j^{(\ell)}= \dfrac{n(n-1)}{2}+1,\qquad 
 \deg W =\dfrac{n(n-1)}{2}.
 $$
Hence 
\begin{gather}\label{FeqC}
F^{(\ell)}(x_1,\dots x_n) = S^{(\ell)}(x_1,\dots x_n) W(x_1,\dots x_n),
\end{gather}
 where 
$S^{(\ell)}(x_1,\dots x_n)$ is a symmetric polynomial with $\deg S^{(\ell)} = 1$, i.e.
 \begin{equation*}
 S^{(\ell)}(x_1,\dots,x_n)=A^{(\ell)}\left(\sum_{i=1}^n x_i+C^{(\ell)}\right),
 \end{equation*}
 with the coefficients $A^{(\ell)},\,C^{(\ell)}\in \Bbb R$   to be determined. 
 
 To find $A^{(\ell)}$, we regard the right-hand side of \eqref{FeqC} as a polynomial of the single variable $x_n$, treating all other variables as parameters. 
 In this case, its coefficient  at $x_n^n$ equals
 $
 A^{(\ell)}\prod_{1\le i<k\le n-1}(x_k-x_i).
 $
 On the other hand, if we determine this coefficient from the left-hand side of \eqref{FeqC}, i.e.  directly 
 from the definition \eqref{sumFj} of $F^{(\ell)}(x)$, we arrive at
 the expression
 $
-(1/B)\prod_{1\le i<k\le n-1}(x_k-x_i).
 $
 By equating these two expressions, we infer 
 \begin{gather}\label{A}
 A^{(\ell)}=-1/B.
 \end{gather} 
 
To proceed further (in particular, to determine the coefficient $C^{(\ell)}$), we divide our analysis into two cases.
 \smallskip
 
\noindent \emph{Case (i): $\ell\ne 0$.} We have
  $$\forall \ell\not=0\ \forall j\in\{1,\dots,n\}:\quad
 F_j^{(\ell)}(s_1,\dots,s_{\ell-1},0,s_{\ell+1},\dots, s_n)=0.
 $$
 This follows from the fact (see \eqref{Fj}) that
 $F_\ell^{(\ell)}(x_1,\dots,x_n)$ is divisible by $s_0-x_\ell$,
 and $F_j^{(\ell)}(x_1,\dots,x_n)$ is divisible by $s_j-x_j$ for every $j\neq \ell$.
 Then,  using \eqref{FeqC}--\eqref{A}, we get:
 $$
  0=F^{(\ell)}(s_1,\dots,s_{\ell-1},0,s_{\ell+1},\dots, s_n)=\frac1B\left(\sum_{\substack{i=1\\i\ne \ell}}^n s_i+C^{(\ell)}\right)(-1)^{\ell}\prod_{i=1}^n s_i\prod_{\substack{0\le i<k\le n\\i\ne \ell,k\ne \ell}}(s_k-s_i),
  $$
  and, since last two products above are non-zero, we conclude that
  $$
 C^{(\ell)}=-\sum_{\substack{i=1\\i\ne \ell}}^n s_i,\quad \forall \ell \in \{1,\dots,n\}.
$$
Hence, one has
$$
F^{(\ell)}(x_1,\dots, x_n)=-\frac1B\left(\sum_{i=1}^n x_i-\sum_{\substack{i=1\\i\ne \ell}}^ns_i\right)\prod_{1\le i<k\le n}(x_k-x_i),
\quad \forall \ell \in \{1,\dots,n\}.
$$
Substituting the above equality into \eqref{eq::wQsl}, we obtain
\begin{align*}
\widehat{Q}(s_\ell)
=\prod_{i=1}^n(t_i-s_\ell)\left(\prod_{1\le i<k\le n}(t_k-t_i)+\frac1B\left(\sum_{\substack{i=1\\i\ne \ell}}^ns_i-\sum_{i=1}^n t_i\right)\prod_{1\le i<k\le n}(t_k-t_i)\right),\;
\forall \ell \in \{1,\dots,n\},
\end{align*}
whence
\begin{gather}\label{Qsl1}
Q(s_\ell)=\dfrac{\widehat{Q}(s_\ell)}{\prod_{1\le i<k\le n}(t_k-t_i)}=\prod_{i=1}^n(t_i-s_\ell)\left(1+\dfrac1B\left(\sum_{\substack{i=1\\i\ne \ell}}^ns_i-\sum_{i=1}^n t_i\right)\right),\; \forall \ell \in \{1,\dots,n\}.
\end{gather}

\noindent \emph{Case (ii): $\ell=0$.} Then we have
  $$
\forall j\in\{1,\dots,n\}:\quad F_j^{(0)}(s_1,\dots, s_n)=0,
 $$
 which follows from the fact (see \eqref{Fj}) that  $F_j^{(0)}(x_1,\dots, x_n)$ is divisible by $s_j-x_j$.
 Then,  using \eqref{FeqC}--\eqref{A}, we obtain:
 $$
  0=F^{(0)}(s_1\dots, s_n)=-\frac1B\left(\sum_{\substack{i=1}}^n s_i+C^{(0)}\right) \prod_{1\le i<k\le n}(s_k-s_i),
  $$
  and, since last product above is non-zero, we conclude that
  $$
 C^{(0)}= -\sum_{i=1}^n s_i,
$$
Hence, one has
$$
F^{(0)}(x_1,\dots, x_n)=-\frac1B\left(\sum_{i=1}^n x_i-\sum_{i=1}^ns_i\right)\prod_{1\le i<k\le n}(x_k-x_i).
$$
Substituting the above equality into \eqref{eq::wQsl}, we obtain
\begin{align*}
\widehat{Q}(s_0)=\prod_{i=1}^n(t_i-s_0)\left(\prod_{1\le i<k\le n}(t_k-t_i)+\frac1B\left(\sum_{i=1}^ns_i-\sum_{i=1}^n t_i\right)\prod_{1\le i<k\le n}(t_k-t_i)\right),
\end{align*}
whence,
\begin{gather}\label{Qsl2}
Q(s_0)=\dfrac{\widehat{Q}(s_0)}{\prod_{1\le i<k\le n}(t_k-t_i)}=\prod_{i=1}^n(t_i-s_0)\left(1+\dfrac1B\left(\sum_{i=1}^ns_i-\sum_{i=1}^n t_i\right)\right).
\end{gather}

The statement of the lemma follows from \eqref{Qsl1}--\eqref{Qsl2}.
\end{proof}

\begin{lemma}\label{lemma:interlace}
	Let $B>0$.
	Let $s_i$, $j=0,\dots,n$ and $t_j$, $j =1,\dots, n+1$ be given  numbers satisfying \eqref{st}.
	Define the polynomial  $Q(x)$  of degree $n$   by
	\eqref{Qdef}.
	Let $a_j$, $j=1,\dots,n$, be the roots of $Q$, ordered according to increasing absolute value. Then  the property
	\begin{gather}
	\label{a:in:ts}
	t_j<a_j<s_j\quad \text{for all }j =1,\dots, n,  
	\end{gather}
	if fulfilled if and only if  the  inequality  
	\begin{gather}\label{Bcond}
	\sum_{j=1}^n t_j<B+\sum_{j=1}^{n-1} s_j
	\end{gather}
	holds true.
\end{lemma}

\begin{proof}Recall that the polynomial $P(x)$ is given by \eqref{Pdef}.
Since $Q(t_j)=- {P(t_j)}/{B}$ for $j=1,\dots, n$ (see Remark~\ref{rem:alter})
and $\sgn P(t_j)=(-1)^{j}$, $j=1,\dots, n$ (this follows easily from \eqref{st}), then
	\begin{equation}\label{eq::sgnQti}
		\sgn Q(t_j)=  (-1)^{j-1},\quad j=1,\dots,n,
	\end{equation}
	whence, at the endpoints of each interval $(t_j,t_{j+1})$, $j=1,\dots,n-1$, the function $Q(x)$ is non-zero and takes opposite signs.
	Consequently, each interval $(t_j,t_{j+1})$, $j=1,\dots,n$ contains at least one root of $Q(x)$.
	Furthermore, since 
	$$
	\sgn Q(x)=(-1)^n \mbox { for sufficiently large }x,
	$$  
	then the semibounded interval $(t_n,\infty)$ also contains a root of $Q(x)$.
	From the above observations, together with the fact that 
	$a_j$, $j=1,\dots,n$, are the roots of $Q$ (recall that $\deg Q = n$) ordered by increasing
	absolute value, we conclude that
	$$
	t_j<a_j<t_{j+1}\quad \text{for all } j=1,\dots,n-1,\qquad t_n<a_n,
	$$
	and then the property \eqref{a:in:ts} holds  if and only if 
	\begin{gather}\label{sgnQ}
	\sgn Q(t_j)=-\sgn Q(s_j)\ne 0\quad \text{for all } j=1,\dots,n.
	\end{gather} 
	By \eqref{eq::sgnQti}, the condition 
	\eqref{sgnQ} is equivalent to 
	\begin{gather}\label{sgnQ1}\sgn Q(s_j)=(-1)^j\quad \text{for all } j=1,\dots,n .
	\end{gather}
	
	To investigate $\sgn Q(s_j)$ we use Lemma~\ref{lemma:Q}. Namely, since $\sgn \left(\prod_{i=1}^n(t_i-s_j)\right)=(-1)^{j}$ (this follows from \eqref{st}),
	then by  \eqref{eq::Qslstatement},
	condition \eqref{sgnQ1} holds if and only if
	$$
	1+\dfrac{1}{B}\left(\sum_{\substack{i=1\\i\ne j}}^n s_i-\sum_{\substack{i=1}}^n t_i\right)>0
	\quad \text{for all } j=1,\dots,n ,
	$$
	which is, obviously, equivalent to 
	\begin{equation}\label{sllessB}
		\sum_{i=1}^n t_i<B+\sum_{\substack{i=1\\i\ne j}}^n s_i
		\quad \text{for all } j=1,\dots,n .
	\end{equation}
	Now, since 
	$$
	\sum_{\substack{i=1\\i\ne j}}^n s_i\ge  \sum_{i=1}^{n-1} s_i
	\quad \text{for all } j=1,\dots,n 
	$$  due to \eqref{st}, then the statement \eqref{sllessB} is equivalent to 
	\eqref{Bcond}. The lemma is proven.
\end{proof}

\section{Properties of $\sigma(\L)$}
\label{sec4}

In Theorem~\ref{th1}, we computed the spectrum of the operator
$\L$ \eqref{L:action}. In this section, we establish several additional useful relations
between the parameters $a_j, b_j$, which determine the coefficients of the operator, and the quantities $s_j, t_j$, which represent the endpoints of the spectral bands. These results will be used later in the solution of the inverse problem.
 
Recall that the function $F$ is given by \eqref{F}.

\begin{lemma}\label{lemma:properties}
	Let $a_j$, $j=1,\dots,n$, and 	$b_j$, $j=0,\dots,n$, be positive numbers,
	moreover, let $a_j$ satisfy  \eqref{aa}.
	Let
	$s_j$, $j=0,\dots,n$, and $t_j$, $j=1,\dots,n+1$, denote
	the roots of the equations $F(x)=0$ and $F(x)=4b_0$, respectively,
	enumerated in ascending order (that is, the interlacing property \eqref{interlace} is fulfilled).
	Then the following holds:
	\begin{itemize}
		\item[(i)] The numbers $a_j,\ i=1,\dots, n,$ are the roots of the polynomial $Q(x)$ defined by \eqref{Qdef}, with $B=4b_0$.\smallskip 
	 
	 	\item[(ii)] 
	 The   upper edge of the spectrum $t_{n+1} $ satisfies
	 \begin{gather}\label{T}
	 t_{n+1}=	4b_0+\sum_{i=1}^{n}(s_i-t_i).
	 \end{gather}

	\end{itemize}
\end{lemma}

\begin{proof}
(i)
We define the function $R(x)$ by
$$R(x)\ceq \sum_{j=1}^n\frac{b_j}{a_j-x}+1.$$
One has $F(x)=x R(x)$, whence 
\begin{gather}\label{Rzeros}
R(s_j)=0,\quad j=1,\dots,n.
\end{gather}
	Reducing $R(x)$ to a common denominator, we get
	$$
	R(x)=\dfrac{R_1(x)}{R_2(x)},
	$$
	where $$R_2(x)=\prod_{j=1}^n (a_j-x),\qquad R_1(x)=R(x)R_2(x)=\prod_{j=1}^n (a_j-x)+ \sum_{j=1}^n  b_j \prod_{\substack{k=1\\k\ne j}}^n(a_k-x).
	$$
	Then, due to \eqref{Rzeros}, we have 
	\begin{gather}\label{R1zeros}
	R_1(s_j)=0,\quad j=1,\dots,n. 
	\end{gather}
Using \eqref{R1zeros} and the facts that $\deg R_1(x)=n$, the leading coefficient of $R_1(x)$ is $(-1)^n$, the numbers $s_j$ are pairwise distinct, and $s_0=0$, we infer
	$$
	R_1(x) =\prod_{i=1}^n (s_i-x)=-x^{-1}P(x),
	$$
	where the polynomial $P(x)$ is given in \eqref{Pdef}.
	Then we have  $R_2(x)=-\dfrac{P(x)}{xR(x)}$,
	whence, taking into account that $t_j\,\ j=1,\dots, n+1.$ are pairwise distinct  solutions of $x R(x)=4b_0$, we get
	\begin{equation}\label{eq::R2ti}
		R_2(t_j)=-\dfrac{P(t_j)}{t_jR(t_j)}=-\dfrac{P(t_j)}{4b_0},\quad j=1,\dots, n+1.
	\end{equation} 
	Thus $R_2(x)$  
	\begin{itemize}
		\item[1.] is a polynomial of degree $n$,
		\item[2.] its leading coefficient is $(-1)^n$,
		\item[3.] its graph passes through the $n$ distinct points\,\footnote{The graph of  $R_2(x)$ passes also through the point
			$\left(t_{n+1},  \dfrac{-P(t_{n+1})}{4b_0}\right)$ (see \eqref{eq::R2ti}), but this fact is not used in the proof of (i).
			} $$\left(t_j,  \dfrac{-P(t_j)}{4b_0}\right),\quad  j =1,\dots, n.$$
	\end{itemize}
	However, the polynomial $Q(x)$ satisfies the same properties 1.-3. -- see Remark~\ref{rem:alter}.
	 Since the properties 1.-3. determine the polynomial uniquely, we conclude that
	\begin{equation}\label{eq::Q=R2}
		\prod_{j=1}^n (a_j-x)=R_2(x)=Q(x),
	\end{equation}
	which proves the statement (i). \smallskip
	
	\noindent (ii)  
	We consider the following polynomial of degree $n+1$:
	$$T(x)=Q(x)+\dfrac{P(x)}{4b_0}.$$
	The leading coefficient of $T(x)$ is $(-1)^{n+1}/(4b_0)$, its free
	coefficient is $Q(0)$, and   
 the pairwise distinct numbers $t_i$, $i=1,\dots, n+1$ are the roots of $T(x)$, which follows from  \eqref{eq::R2ti}--\eqref{eq::Q=R2}. Then by Vieta's theorem we obtain
	$
	\prod_{i=1}^{n+1} t_i= 4b_0 Q(0),
	$
	whence, 
	\begin{equation}\label{eq::TQ0}
		t_{n+1}=\dfrac{4b_0 Q(0)}{\prod_{i=1}^n t_i}.
	\end{equation}
	Finally, using Lemma \ref{lemma:Q} and taking into account that
	$s_0=0$, one has
	\begin{align} 
		Q(0)
		=\left(\prod_{i=1}^n t_i\right)\left(1+\dfrac1{4b_0}\left(\sum_{i=1}^n s_i-\sum_{i=1}^nt_i\right)\right)\label{Q0}
	\end{align}
	Substituting \eqref{Q0} into \eqref{eq::TQ0} we 
	arrive at the formula \eqref{T}.
 	The lemma is proven.
\end{proof}

\section{Inverse problem}
\label{sec5}

In this section we solve the inverse problem:
for given $s_j$ and $t_j$, we construct an operator $\L$ defined by \eqref{L:action}
such that its spectrum admits the representation \eqref{th1:eq}
with exactly these prescribed values of $s_j$ and $t_j$.
Namely, we have the following result.

\begin{theorem}\label{th2}
	Let $s_j$, $j=0,\dots,n$, and $t_j$, $j=1,\dots,n+1$, be given numbers satisfying \begin{gather}\label{st+}
0=s_0< t_1,\qquad t_j< s_j< t_{j+1},\ j=1,\dots,n.
\end{gather}
	Using these numbers, we define   $b_0>0$ by
	\begin{gather}\label{b0}
	b_0\ceq \dfrac14\sum_{j=0}^{n}\left(t_{j+1}-s_j\right),
	\end{gather}
	and the polynomial $Q(x)$ of degree $n$ by 
\begin{gather}\label{Qdef:new}
Q(x)=\prod_{i=1}^n(t_i-x)-{1\over 4 b_0}\sum_{j=1}^n \left( \prod_{m=0}^n (s_m-t_j) \prod_{\substack{i=1\\i\ne j}}^n\dfrac{x-t_i}{t_j-t_i}\right).
\end{gather}	
Let $a_j$, $j=1,\dots,n$, be the roots of $Q$, ordered according to increasing absolute value. 	
\smallskip

Then the following holds:	

\begin{itemize}	
	
\item[(i)] The numbers $a_j$, $j=1,\dots,n$  are positive and pairwise distinct, moreover, they satisfy
the interlacing property 
\begin{gather}\label{a:in:ts+}
t_j<a_j<s_j,\quad j=1,\dots,n.
\end{gather}

\item[(ii)] Using the above given  $s_j$, $j=1,\dots,n$, and the previously defined   $a_j$, $j=1,\dots,n$, we define  the numbers $b_j$, $j=1,\dots,n$, by
	\begin{gather}\label{bj}
		b_j\ceq (s_j-a_j)\prod_{\substack{i=1\\i\ne j}}^n\dfrac{s_i-a_j}{a_i-a_j},
		\quad j=1,\dots,n.
	\end{gather}
	Then $b_j$ are positive, moreover, one has
\begin{gather}\label{SLAE}
 	\sum_{j=1}^n\frac{b_j}{a_j-s_i}+1=0,\quad i=1,\dots,n.
\end{gather}\smallskip
	
\item[(iii)] Using the numbers $a_j$, $j=1,\dots,n$, and $b_j$, $j=0,\dots,n$, defined   above, we construct the graph $\Gamma$ by \eqref{graph:structure}--\eqref{mu:m}, with the functions $\mu$ and $m$ given by
	$$
	\mu_j=b_j,\ j=0,\dots,n,
	\qquad m_0=1,\qquad
	m_j=b_j/a_j,\ j=1,\dots,n.
	$$
	Then the spectrum of the operator $\L$, defined by \eqref{L:action}, admits a representation 
\begin{gather*}
\sigma(\L)=[0,t_{n+1}]\setminus\left( \bigcup_{j=1}^n (t_j,s_j)\right),
\end{gather*}	
	where
	$s_i$, $j=1,\dots,n$,  $t_j$, $j=1,\dots,n+1$, are exactly the numbers given above.

\end{itemize}
\end{theorem}

\begin{proof}
(i) By Lemma~\ref{lemma:interlace}, the numbers $a_j$ satisfy \eqref{a:in:ts+}
provided  \eqref{Bcond} holds with
$
B=4b_0=\sum_{j=0}^{n} (t_{j+1}-s_j),
$
that is,
\begin{gather*}
	\sum_{j=1}^n t_j
	<
	\sum_{j=0}^{n} (t_{j+1}-s_j)
	+
	\sum_{j=1}^{n-1} s_j.
\end{gather*}
Since $s_0=0$, the above inequality is equivalent to $s_n<t_{n+1}$, which indeed holds
by \eqref{st+}. Hence, the above choice of $a_j$, $j=1,\dots,n$,
is admissible: all $a_j$ have the   property \eqref{a:in:ts+}, in particular, they are positive and satisfy \eqref{aa}.\smallskip

\noindent(ii) It follows from \eqref{st+} and already established \eqref{a:in:ts+} that
$$
\sgn(s_i-a_j)=\sgn(a_i-a_j)
\quad \text{for all } i\neq j,
\qquad
\sgn(s_j-a_j)=1.
$$
Hence, the numbers $b_j$, $j=1,\dots,n$, defined by \eqref{bj} are positive.

Regarding \eqref{SLAE} as a system of $n$ linear algebraic equations
 for the unknowns $b_j$, $j=1,\dots,n$, it remains to show that
 its solution is given by \eqref{bj}.
 
 	We prove the above statement by induction. For $n=1$, the claim is
 	obvious. Assume that it has been established for $n=N-1$.
 	We now prove it for $n=N$.
 	Multiplying the $i$-th equation in \eqref{SLAE} ($i=1,\dots,N$)
 	by $a_N-s_i$ and then subtracting the $N$-th equation from the
 	first $N-1$ equations we arrive at the new system of $N-1$ equations
 	\begin{gather*}
 		 \sum_{j=1}^{N-1}{\widehat b_j\over
 			a_j-s_i}+1,\quad  i=1,\dots, N-1,
 	\end{gather*}
 	where the new unknowns $\widehat b_j$, $j=1,\dots, N-1$ are
 	expressed in terms of $b_j$ as follows,
 	\begin{gather}\label{new}
 		\widehat b_j=b _j\ds{a_N-a_j\over s_N-a_j},\quad j=1,\dots, N-1
 	\end{gather}
The system \eqref{new} coincides with the system \eqref{SLAE} for $n=N-1$.
Therefore, by the induction hypothesis, we obtain
 	\begin{gather}\label{syst_sol_ind}
 		\widehat b_j= (s_j-a_j)\prod_{\substack{i=1\\i\ne j}}^{N-1}\dfrac{s_i-a_j}{a_i-a_j},\qquad j=1,\dots,N-1.
 	\end{gather}
 	Combining \eqref{new} and \eqref{syst_sol_ind}, we conclude that $b_j$, $j=1,\dots,N-1$, satisfy \eqref{bj}. The validity
 	of \eqref{bj} for $b_N$ follows   from the symmetry of
 	the system \eqref{SLAE}.
\smallskip
 
(iii) By Theorem \ref{th1} (taking into account \eqref{ab:notations})  we have
\begin{equation*}
	\sigma(\L)=\displaystyle\bigcup_{j=1}^{n+1}[\wt s_{j-1},\wt t_{j}],
\end{equation*}
where $\wt s_0=0$, and   $\wt s_1< \dots<\wt s_{n}$  satisfy
\begin{equation}\label{wtsj:eq}
	\sum_{j=1}^n\dfrac{b_j}{a_j-\wt s_i}+1 =0,\quad i=1,\dots,n,
\end{equation}
while $\wt t_1< \dots<\wt t_{n+1}$ satisfy
\begin{equation*}
	\wt t_i\left(\sum_{j=1}^n\dfrac{b_j}{a_j-\wt t_i}+1\right)=4b_0,\quad i=1,\dots,n+1.
\end{equation*}

We define the polynomial $\wt Q(x)$ of degree $n$  
by \eqref{Qdef:new}, but with $\wt s_j$, $\wt t_j$ instead of $ s_j$, $ t_j$. respectively.
We also introduce two polynomials of degree $n+1$,
$$
T(x)\ceq Q(x)+\dfrac{P(x)}{4b_0}\text{\quad and\quad }
\wt T(x)\ceq \wt Q(x)+\dfrac{\wt P(x)}{4b_0},
 $$
where $P(x)\ceq\prod_{j=0}^n (s_j-x)$ and $\wt P(x)\ceq\prod_{j=0}^n (\wt s_j-x)$. 
The numbers $t_j$, $j=1,\dots,n$ (respectively, $\widetilde t_j$, $j=1,\dots,n$) are the roots of $T(x)$ (respectively, $\widetilde T(x)$); this follows immediately by substituting these numbers into the corresponding polynomials, see Remark~\ref{rem:alter}.
Furthermore, the number $\wt t_{n+1}$ is also the root of $\wt T(x)$ -- see the beginning of the 
proof of Part (ii) of Lemma~\ref{lemma:properties}.

Using \eqref{SLAE} (resp. \eqref{wtsj:eq}), we conclude that $s_j$ (resp. $\wt s_j$)  are all $n+1$ roots of the rational function $F(x)$ \eqref{F} (see Figure \ref{fig2}). Hence, since
$s_j$ and $\wt s_j$ are enumerated in ascending order, we conclude that
\begin{equation}\label{s:wts}
s_j = \wt s_j,\quad j=1,\dots,n.
\end{equation}
Recall the numbers $a_j$ are  the roots of the polynomial $Q(x)$ \eqref{Qdef:new} and (see \eqref{a:in:ts+}) they are pairwise distinct.
By Lemma~\ref{lemma:properties}(i) $a_j$ are also the roots of the polynomial $\wt Q(x)$.
Both $Q$ and $\wt Q$ are of degree $n$ and their leading coefficient 
is $(-1)^{n}$. Hence the polynomials $Q(x)$ and $\wt Q(x)$ do coincide;
consequently, by \eqref{s:wts}, we have
$$
T(x)=\wt T(x),\quad \forall x,
$$
whence 
\begin{gather}\label{inclusion}
\{t_1<\dots<t_n\}\subset \{\wt t_1<\dots<\wt t_n<\wt t_{n+1}\}.
\end{gather}
However, since $\wt t_j<a_n$ for $j=1,\dots,n$ and $\wt t_{n+1}>a_n$
(in fact, \eqref{interlace} holds with $\wt t_j$ in place of $t_j$),
while, by \eqref{a:in:ts+}, we have
$t_j<a_n$ for $j=1,\dots,n$, it follows from \eqref{inclusion} that
\begin{equation}\label{t:wtt}
t_j = \wt t_j,\quad j=1,\dots,n.
\end{equation}

Finally, using Lemma~\ref{lemma:properties}(ii) and  
\eqref{s:wts}, \eqref{t:wtt},  
we get
\begin{gather}\label{T:wtT}
 \wt t_{n+1}=  4b_0+\sum_{i=1}^{n}(\wt s_i-\wt t_i)=
 4b_0+\sum_{i=1}^{n}(  s_i-  t_i)=t_{n+1},
\end{gather}
where on the last step we substitute the expression \eqref{b0} for $b_0$
and use the fact that $s_0=0$.

The desired result follows from \eqref{s:wts}, \eqref{t:wtt}, \eqref{T:wtT}.
The theorem is proven.
\end{proof}

\begin{remark}
It follows immediately from the theorem above and its proof that, given the graph $\mathbb Z$ equipped with the constant positive edge weight $\mu_0$ and vertex weight $m_0\equiv 1$, one can ``decorate'' it by attaching periodically arranged ``bristles'' with suitably chosen weights so that the spectrum of the Laplacian on the resulting weighted graph possesses $n$ gaps, which coincide with the prescribed intervals $(t_j,s_j)$. The upper edge of the spectrum of this operator will be  given by 
$$
t_{j+1}=4 {\mu_0}+\sum_{j=1}^{n}(s_j-t_j).
$$
\end{remark}

\begin{remark}\label{rem:comparison}
At the end of this section, we compare our results
with the one obtained in \cite{BK15} for quantum graphs.	
In that paper, a $\mathbb Z^n$-periodic metric graph $\Gamma$
and a family of Hamiltonians  $\mathscr{A}_\varepsilon$
on $\Gamma$ acting as $-\frac1\varepsilon\frac{\mathrm{d}^2}{\mathrm{d} x^2}$ on the edges of $\Gamma$, subject to appropriate (the so-called $\delta'$-type)   conditions at the vertices is considered; where $\varepsilon>0$ is a small parameter. It was shown that, as $\varepsilon\to0$, the spectrum of $\mathscr A_\varepsilon$ possesses at least $n$ gaps ($n\in\mathbb N$ being prescribed in advance). The first $n$ gaps converge as $\varepsilon\to 0$ to certain intervals whose location can be effectively controlled by an appropriate choice of the geometry of $\Gamma$ (resembling the decorated graphs from \cite{SA00,Ku05}) and of the coupling constants appearing in the interface conditions. Any remaining gaps escape to infinity as $\varepsilon\to0$.

The inverse problem solved in \cite{BK15} is considerably simpler than the one studied in the present work. The right endpoints of the gaps are given by the roots of an \eqref{F}-type function $F(x)$, whereas the left endpoints are expressed explicitly in terms of those coupling constants. Thus, the left gap edges can be controlled very easily, while controlling the right edges requires solving an \eqref{SLAE}-type system. This is substantially simpler than the situation in the present paper, where the right edges are determined by the equation \eqref{t:eq}.

Finally, we note that the result of \cite{BK15} was further refined in \cite{Kh20}. In particular, it was shown how to ensure, for fixed (sufficiently small) $\varepsilon$, the exact coincidence of the left endpoints of the first $n$ spectral gaps with prescribed values, however, the coupling constants appearing in the interface conditions are not determined explicitly in this approach.
\end{remark}

\section{Acknowledgements}

A. K. is grateful to Excellence Project FoS UHK 2204/2025-2026 for the   support.

\end{document}